\documentclass{article}

\usepackage[utf8]{inputenc}
\usepackage{amsmath,bm,hyperref}
\usepackage{amsfonts,amssymb,amsthm,hyperref,enumitem,comment,color}
\usepackage[title]{appendix}

\newtheorem{theorem}{Theorem}
\newtheorem{lemma}{Lemma}
\newtheorem{remark}{Remark}
\newtheorem{definition}{Definition}
\newtheorem{corollary}{Corollary}
\newtheorem{conjecture}{Conjecture}

\newtheorem{ex}{Example}
\def\Dres{{\rm DRes}}

\def\Ker{{\rm Ker}}

\def\BC{{\tt BC}}
\def\m0{{\bf 0}}

\def\cE{{\mathcal{E}}}

\def\Keywords{{\bf Keywords:} }
\def\Classification{{\bf MSC : }}
\title{Burchnall-Chaundy polynomials for {matrix ODOs}\\ and {Picard-Vessiot Theory} }

\author{Emma Previato, Sonia L. Rueda and Maria-Angeles Zurro}

\begin{document}

\maketitle

\begin{abstract}
{Burchnall and Chaundy showed that if two ordinary differential operators (ODOs)  $P$, $Q$ with analytic coe\-ffi\-cients commute then there exists a polynomial $f(\lambda ,\mu)$ with complex coefficients such that $f(P,Q)=0$, called the BC-polynomial. This polynomial can be computed using the differential resultant for ODOs. In this work we extend this result to matrix ordinary differential operators, MODOs. Our matrices have entries in a differential field $K$, whose field of constants $C$ is algebraically closed and of zero characteristic. We restrict to the case of order one operators $P$, with invertible leading coefficient. We define a new differential elimination tool, the matrix differential resultant. We use it to compute the BC-polynomial $f$ of a pair of commuting MODOs, and we also prove that it has  constant coefficients. This resultant provides the necessary and sufficient condition for the spectral problem $PY=\lambda Y \ , \ QY=\mu Y$ to have a solution. Techniques from differential algebra and Picard-Vessiot theory allow us to describe explicitly isomorphisms between commutative rings of MODOs $C[P,Q]$ and a finite product of rings of irreducible algebraic curves. }
\end{abstract}


\Classification{13N10; 13P15; 14H70} 


\section{Introduction}\label{Introduction}

The main contribution of this paper is the construction of a new differential elimination tool, a differential resultant for matrix ordinary differential operators (MODOs). Our main goal is to prove that this tool provides  an effective criterion to guarantee  the solvability of the 
eigenvalue problem for commuting MODOs
\begin{equation}\label{eq-eigenProblem}
    L Y =\lambda Y \quad , \quad B Y = \mu Y \ .
\end{equation}
We restrict to the case where $L$ is monic and has order one, since 
according to Wilson \cite{wilson}, this situation is interesting because $L$ does have order $1$ in practically all the most interesting examples \cite{ZS}, see also \cite{Dubrovin}. {Furthermore, it should be noted that by means of the Cyclic Vector Lemma, there is a correspondence between systems of order $1$ and size $\ell \times \ell$ and differential operators of order $\ell$, see for example Churchill and Kovacic \cite{ChurchillKovasic}, or Katz \cite{Katz1987}.}

\medskip

Differential resultants were first defined for ordinary differential operators, as the natural generalization to a non commutative environment of the algebraic resultant of two univariate polynomials, see for instance \cite{Ch}. A few years ago the theory of differential resultants was formalized for multivariate differential polynomials and reviewed in two recent reports \cite{McW} and \cite{LY}. See also \cite{KasmanPreviato2001, KasmanPreviato2010} for differential resultants in the case of partial differential operators. We provide here the first definition of a differential resultant for MODOs. Every previously existing notion of differential resultant provides a condition on the coefficients of the differential operators (or differential polynomials) that guarantees the existence of common nontrivial solutions. The tool we  develop here  provides such a condition in the case of MODOs, but many difficulties emerge in trying to extend previous methods to the present situation. These are a consequence of two main facts:  matrix coefficient rings are non commutative and, in addition, rings of MODOs are not euclidean domains. 

\medskip

Throughout this paper we use the language of differential algebra and Picard-Vessiot theory.  For the main definitions and notation we recommend the fo\-llowing references \cite{Ritt}, \cite{Kolchin}, \cite{VPS}, \cite{Ch}, and \cite{MR1999}. The essential terminology has been summarized in the Appendix.
We consider the ring $\mathcal{R}_\ell$ of $\ell\times \ell$ matrices with entries in an ordinary differential field $K$, whose field of constants $C$ is algebraically closed. The derivation of $K$ can be extended to a derivation $D$ on $\mathcal{R}_\ell$. 
The MODOs of this paper belong to the ring of differential operators $\mathcal{R}_\ell [D]$, for details see Section \ref{sec-modos}.

\medskip

In this article, we define in Section \ref{sec-commonSol} (see Definition \ref{def-defDres}) the differential resultant $\Dres (P,Q)$ of two {  MODOs} $P$ and $Q$ in $\mathcal{R}_\ell [D]$, in the case where $P=A_0 +A_1 D$,  with $A_1$ invertible, and prove the following result. 

\ 

{\bf Theorem A. }{\it  
Given {  MODOs} $P$ and $Q$ in $\mathcal{R}_{\ell}[D]$, $P$ of order $1$ with invertible leading coefficient matrix. The following statements hold:
\begin{enumerate}
    \item If there exists a common nontrivial solution in some differential field extension of $K$ of $PY=\overline{0}$ and $QY=\overline{0}$ then $\Dres (P,Q)=0$.
    
    \item {If $P$ and $Q$ commute, and}  {if }  $\Dres (P,Q)=0$, then the matrix differential system $PY=\overline{0}$ , $QY=\overline{0}$, has a solution  $\psi = (\psi_1 , \dots ,\psi_{\ell} )^t$ with every $\psi_i$ in a differential extension of finite algebraic transcendence degree  $\Sigma$ of $K$.
\end{enumerate}
}

{These results motivate the definition of the differential resultant for MODOs in the general case, where $P$ has arbitrary order. For that purpose the Picard-Vessiot theory of MODOs needs further development that could benefit from  results about the factorization of MODOs 
 such as those obtained by A. Kasman in \cite{Kasman2017}.}

The classical problem of describing pairs of commuting differential operators was first studied
by Burchnall and Chaundy \cite{BC1}, \cite{BC2}, and Baker in \cite{Baker1928}, for ODOS. {The work of Grinevich in \cite{Grinevich1987} generalizes this problem to the case of MODOs, assuming that they are in so called {\sl general position} and for matrix coefficients in the field of analytic complex functions.} The effective construction of commuting matrix differential operators has been also addressed in Oganesyan's article \cite{Oganesyan2017}.

This paper contributes to a  generalization to the case of MODOs, of the famous results by Burchnall and Chaundy in \cite{BC1}, which establish a correspondence between rings of commuting ordinary differential operators and planar algebraic curves, see also \cite{K78}. For a commuting pair $L$, $B$ of ODOs, it is easy to observe the existence of a polynomial $h(\lambda,\mu)$ with constant coefficients such that $h(L,B)=0$: Burchnall and Chaundy showed that the opposite is also true \cite{BC1, BC2} in the specific situation they studied, see \cite{GZ} for recent contributions on this matter. This is the defining polynomial of a plane algebraic curve $\Gamma$, commonly known as {\sl the spectral curve}, and it can be computed by means of the differential resultant of $L-\lambda$ and $B-\mu$. {These facts are an important motivation to develop differential resultants for MODOs  to study the spectral problem \eqref{eq-eigenProblem}}. 

Next we consider $P= L-\lambda$ and $Q=B-\mu$ as MODOs with matrix coefficients in the differential field $K(\lambda ,\mu)$, for algebraic variables $\lambda$ and $\mu$, whose field of constants is $C(\lambda,\mu)$. 
We generalize Previato's theorem on differential resultants for ODOs, see \cite{Prev} and \cite{MRZ1}, to prove the next analogous result for MODOs, {showing that the differential resultant of $L-\lambda$ and $B-\mu$ is a polynomial with differentially constant
coefficients.}

\medskip

{\bf Theorem B. }{\it 
 \ Let us consider matrix differential operators $L$ and $B$ in $\mathcal{R}_{\ell}[D]$, and assume that $L$ has order one with invertible leading coefficient. 
If $L$ and $B$ commute then the  differential resultant  
\begin{equation}\label{def-spCurve}
    f(\lambda ,\mu) = \Dres(L-\lambda,B-\mu)
\end{equation}
is a polynomial in $C [\lambda ,\mu ]$.
}

\medskip

Coming back to Theorem A, for a commuting pair $L$, $B$, the spectral problem \eqref{eq-eigenProblem} has a nontrivial solution, for $\lambda=\lambda_0$ and $\mu=\mu_0$, if and only if $f(\lambda_0,\mu_0)=0$, see Corollary \ref{cor-coupled}. Thus \eqref{eq-eigenProblem} is in fact a coupled problem since $\lambda$ and $\mu$ are not free parameters. The algebraic curve $\Gamma$ defined by $f(\lambda,\mu)=0$,
{ which guarantees the solvavility} of the eigenvalue problem,  is the so-called {\sl spectral curve}, see \cite{Grinevich1987}. 
{In other words,  each point of $\Gamma$ provides a spectral problem associated to a pair of commuting operators, $L, B$ that admits a common solution.}

\medskip

The present work studies the so called {\sl direct problem} for commutative algebras $C[L,B]$
associated with a commuting pair of MODOs $L, B$, {for the first time} in the case of  an arbitrary differential field $K$. More precisely, given $L,B$ in $\mathcal{R}_\ell [D]$, assuming that $L$ has order one and invertible leading coefficient matrix, we prove in Theorem C a decomposition theorem for the algebra $C[L,B]$ in terms of the irreducible components of the spectral curve in this context. The tool we develop, the differential resultant for MODOs, plays a crucial role. It is important to note that we do not restrict to the case of irreducible curves or nonsingular curves. Working in an arbitrary differential field and using Picard-Vessiot theory allows to {reduce} the hypothesis on the leading coefficient of $L$, in comparison with previous works, see for instance \cite{Grinevich1987}. Similar benefits are expected for $L$ of arbitrary order, after an appropriate differential resultant for MODOs is defined.

We define the {\sl Burchnall-Chaundy ideal of the pair $L,B$} to be the set 
\[
\BC (L,B) :=\{g\in C[\lambda,\mu]\mid g(L,B)=\m0\},
\]
whose elements are {\sl Burchnall-Chaundy (BC) polynomials}, in analogy with the theory of ODOs. We assume that the differential resultant $f$ is a $\BC$ polynomial, being this assumption very likely to happen as explained in Remark \ref{rem-BCpol}. We present this fact as a conjecture, for $K$ an arbitrary differential field, in Section \ref{sec-BC}. Moreover, once $f$ has been computed, its decomposition in irreducible factors $f=h_1^{\sigma_1 } \cdots h_s^{\sigma_s}$ will allow us to give the Algorithm \texttt{BC-generator} to compute a polynomial $F$ such that $\BC(L,B)=(F)$.

We establish the ring structure of the commutative algebra $C[L,B]$ by means of an isomorphism 
$$C[L,B]\simeq\frac{C[\lambda , \mu ] }{\BC(L,B)}.$$ 
{The next decomposition theorem is by itself important and allows to classify the commutative algebras $C[L,B]$ as products of quotient rings related with the irreducible components $\Gamma_i$ of the spectral curve $\Gamma$. }

\ 

{\bf Theorem C. }{\it 
\ Let us consider commuting matrix differential operators $L$ and $B$ in $\mathcal{R}_{\ell}[D]$, and assume that $L$ has order one with invertible leading coefficient. Let $f(\lambda ,\mu) = \Dres(L-\lambda,B-\mu)$ and assume that $f(L,B)=\m0$.
Then there exists a polynomial $F=h_1^{r_1 } \cdots h_s^{r_s}$ that divides $f$ such that $\BC (L,B) =(F)$. Furthermore the following isomorphism can be established
\begin{equation}\label{eq-decomposition}
C[L,B]\simeq 
\frac{C[\lambda , \mu ] }{(h_1^{r_1 })}
\times \cdots \times 
\frac{C[\lambda , \mu ]}{(h_s^{r_s })}    ,
\end{equation}
whose ring structure is componentwise
addition and multiplication.
}

{
These results can be used to classify the commutative algebras $C[L,B]$ in terms of the irreducible components of the spectral curve, more precisely the irreducible factors of its defining polynomial $f$, the differential resultant. In particular, if $\Gamma$ is an irreducible curve then 
\[C[L,B]\simeq \frac{C[\lambda , \mu ]}{(h^{r})}\]
where $h$ is is the unique irreducible factor of $f=h^{\sigma}$, $1\leq r\leq \sigma$. For matrix coefficients of size $\ell=2$ the classification is then clear and can be applied to pairs of operators $L$ and $B$ defining the famous AKNS hierarchy. We finish this paper illustrating our results with a computed example, the first non trivial case of the AKNS hierarchy, since $f(L,B)=\m0$ in this case. All computations were performed in \href{https://www.maplesoft.com/}{Maple} 21.
}

\bigskip

{\it The paper is organized as follows}.  Section 2 includes important notation and definitions, together with the context of the problems studied in relation to the work of other authors {using different approaches}. Section 3 is dedicated to the construction of the differential resultant for MODOs and the proof of Theorem A. The spectral curve $\Gamma$ is defined in Section 4 after proving Theorem B.  {It is shown that problem \eqref{eq-eigenProblem}  is a coupled spectral problem, that admits a solution whenever $\lambda=\lambda_0$ and $\mu=\mu_0$ are related through $f(\lambda_0,\mu_0)=0$}. It is in Section 5 that the ideal of BC polynomials $\BC (L, B)$ is canonically associated to the operators $L, B$. In Theorem \ref{thm-ideal1} we prove that the differential resultant $f(\lambda,\mu)=\Dres(L-\lambda,B-\mu)$ provides a MODO $f(L,B)$, which is zero on the solution space of $L-\lambda$. Under the assumption $f(L,B)=\m0$, in Theorem \ref{thm-ideal} it is proved that the ideal $\BC(L,B)$ is bounded by ideals defined by $f$. Section 6 contains the proof of Theorem C and the Algorithm \texttt{BC-generator} to compute its decomposition as a product of rings associated with irreducible curves. Finally, we illustrate our results in Section 7 by applying them to the AKNS hierarchy, where all hypothesis are fulfilled. 

\medskip

Professor Emma Previato passed away on June 29, 2022 while we were preparing the final version of this article. This work has been carried out under her constant inspiration even after her untimely passing. We sincerely wish her rest in peace, and we fulfill her wish that this work be published in the special volume in honor of H. Flaschka.

\section{Matrix coefficient ODOs}\label{sec-modos}

Let K be a differential field  with derivation $\partial$, whose field of constants $C$ is algebraically closed of characteristic zero. Given $a\in K$ we denote $\partial(a)$ by $a'$. The commutation rule in the ring of differential operators $K[\partial]$ is then defined by $\partial a =a\partial+a'$.  

Let us consider the ring $\mathcal{R}_\ell =M_{\ell}(K)$ of $\ell\times \ell$ matrices with coefficients in $K$. Given $A=(a_{\alpha,\beta})\in\mathcal{R}_\ell $, let us denote by $A'=(a'_{\alpha,\beta})$. Thus we can extend the derivation $\partial$ to a derivation $D$ in $\mathcal{R}_\ell$ as $D(A):=A'$. We will work with  matrix coefficient differential operators as elements of $\mathcal{R}_\ell [D]$, called \emph{ matrix ordinary  differential operators} or MODOs, where the commutation rule is naturally defined by $D A:=AD+A'$. In addition, we will denote by $\m0$ the matrix with all  entries equal to $0$. {Observe that we can identify $\mathcal{R}_\ell [D]$ with $M_{\ell}(K[\partial])$.}

\medskip

The ring of scalar differential operators $K[\partial]$ is embedded in the ring of matrix differential operators $\mathcal{R}_\ell [D]$ by sending a scalar differential operator $\sum a_i \partial^i$ to the matrix differential operator $\sum a_i I_{\ell} D^i$, for the identity matrix $I_{\ell}$ of $\mathcal{R}_\ell $.
In the scalar case, ODOs enjoy many properties that are essential to the classification of commutative rings of ODOs. In the matrix-coefficient case, it is not known whether such properties hold. We list the main ones: 
\begin{enumerate}[label= \Roman* ]
\item \label{item-I}\label{property-P1} The commutator of two operators of orders $m$ and $n$ has order strictly less than $n+m$. 

\item \label{item-II} If two operators commute with an operator $L$ of order at least $1$, then they commute among themselves{, since the centralizer of $L$ is a commutative ring}, \cite{Good}. \label{property-P2}
\end{enumerate}

Observe that, in the matrix case, Property \ref{property-P1} does not hold. For a counterexample, let us consider operators $MD,\ ND$, where $M,\ N$ are two constant non-commuting matrices
\[[MD,ND]=[M,N]D^2+(N'-M')D.\]
One might wonder whether in the algebraic case 
the property might hold; in fact, the assumption in \cite{wilson} is that the leading coefficients are diagonal. In addition, Property \ref{property-P2} fails. In fact, $MD,\ ND$ as above both commute with $I_{\ell} D$.

These observations bring out the intrinsic interest of studying commutative subrings of 
matrix differential operators; {however, given the difficulties that arise, some special assumptions are usually considered. We review below some of the existing literature to highlight the contributions of the present paper.}

\medskip


Observe that the ring $\mathcal{R}_\ell [D]$ of MODOs is included in the algebra  of pseudo-differential operators with matrix coefficients
$$\mathcal{R}_\ell [D ,D^{-1}]:=\left\{\sum_{i=-\infty}^n A_i D^i\mid A_i\in \mathcal{R}_\ell, n\in\mathbb{Z}\right\}.$$
For a pseudo-differential operator $L=\sum_{i=-\infty}^n A_i D^i $, we call $A_n$ its leading coefficient, whenever it is non-zero, and $n$ is called its order. Furthermore, we will say that L is in normal form if $A_{n-1}=0$.

\vspace{0.5cm}

The above construction has been generalized as follows. Given 
$
L=\sum_{i=0}^n  U_i D^i  ,    
$
with $U_i =\displaystyle\left( u_{i, \alpha \beta} \right)_{1\leq \alpha , \beta \leq l}$, if we consider $u_{i,\alpha\beta}$ as differential variables over $C$, we can define the ring of differential polynomials
$$
\mathcal{B}=C\{u_{i,\alpha\beta}\}={C}\left[u^{(j)}_{i,\alpha\beta} \, ; \,  1\le \alpha, \beta\le l, \ 0\le i\le n, \  j\ge 0  \ \right].
$$ 
In \cite{wilson}, Wilson studied {the centralizer $\mathcal{Z}(L)$ of $L$} in $M_l (\mathcal{B})[D ,D^{-1}]$, and it is in this algebra that some of the properties of scalar ODOs persist. In \cite{wilson}, Proposition 2.19, he shows  that $\mathcal{Z}(L)$ is commutative, i.e. that Property \ref{property-P2} holds under the following assumptions:
\begin{enumerate}[label=(\alph*)]
    \item \label{item-a} The leading coefficient $U_n$ is an invertible diagonal matrix,  diag$ (c_1, \ldots ,c_k)$, the $c_\alpha$ being non-zero constants.

    \item \label{item-b} If $c_\alpha= c_\beta$, then $u_{n-1,\alpha\beta} = 0$.
\end{enumerate}
{As Wilson writes, one can conjugate $L$ into its leading term by a suitable `integral operator' (formula (5.1) in \cite{wilson}). This is what gives the affine ring of a curve.
}

\medskip


In Mulase et al. \cite{MuKi}, a correspondence is established between integral algebras of MODOs ({those containing an operator in normal form})  and geometric data {related to an algebraic curve}, assuming that the differential field $K$  is $C((x))$, the field of  formal Laurent series in the variable $x$ with coefficients in $C$. In this framework, there exists a classification of commutative elliptic algebras of MODOs, in the Verdier sense \cite{Ve}. In Mulase et al. \cite{MuLi} an equivalence of categories provides such  a classification: On one hand the category of algebras of commuting MODOs; on the other, coverings of algebraic curves (spectral curves) together with some associated geometric data.

This type of rings occurred in the study of the so-called {\sl inverse spectral problem}  for an algebraic curve with some extra geometric data. The problem was initially studied by Krichever \cite{krichever}. He established a method to use MODOs to construct a pair of matrix operators whose spectral curve is some given curve $\Gamma$, based on the configuration of its points at infinity. {One can observe that the matrices associated to a curve by Krichever's construction are of Wilson's type, \cite{previato}.}

\medskip


Given $L \in  \mathcal{R}_\ell [D]$, 
{one can consider the centralizer $\mathcal{C}(L)$ of $L$ in $ \mathcal{R}_\ell [D]$ or the centralizer 
$Z(L)$ of $L$ in $\mathcal{R}_\ell [D ,D^{-1}]$.
In this paper we study the algebras $C[L,B]$, for $B$ in $\mathcal{C}(L)$. One first observation is that the algebra $C[L,B]$} is not in general a maximal commutative subalgebra of {$\mathcal{R}_\ell [D]$} 
. The following inclusion
\[
C[L,B] \subset \mathcal{C}(L)=\mathcal{R}_\ell [D] \cap Z(L).
\]
in general proper, suggests the interest in studying  the centralizer $\mathcal{C} (L)$ of $L$ in $\mathcal{R}_\ell [D]$.
{The results of K. Goodearl in \cite{Good} concerning centralizers of ODOs have proven to be important for the effective computational approaches in  \cite{PRZ2019, MRZ1, MRZ2}.}
{In this sense, a generalization to the case of MODOs of Goodearl's results would give an effective description of  $\mathcal{C} (L)$, but we emphasize that we do not pursue such a description  in this paper. }

\medskip

{In this work,} as in the scalar case, we will use the concept of differential resultant to establish a correspondence between pairs of commuting matrix differential operators $L$, $B$ and algebraic curves \cite{BC1}. {More precisely, we will assign to each commutative algebra $C[L,B]$ a plane algebraic curve, the \emph{spectral curve}.} {This is the goal of the second part of this paper from Section \ref{sec-spectral curves} onward}.
For this purpose, let us consider algebraic variables $\lambda$ and $\mu$ with respect to $\partial$. To complete this brief overview of the general situation for MODOs, we would like to indicate that by using Picard-Vessiot extensions, we obtain a representation of the centralizer $\mathcal{C} (L)$. Let $\cE$ be a Picard-Vessiot field of the  differential equation $LY =\lambda Y$, for details see Remark \ref{rem-PVE}. Since $L-\lambda=L-\lambda I_{\ell}$ and $B-\mu=B-\mu I_{\ell}$ commute, the operator $B-\mu$ acts linearly on the solution space of $L-\lambda$, and this action provides an $\ell\times \ell$ matrix $M(L-\lambda ,B-\mu )$, as constructed in Section \ref{sec-spectral curves}. Moreover, the null space of  $M(L-\lambda ,B-\mu )$ provides a representation of the space of common solutions at each point of the spectral curve.

\medskip

{
It was E. Previato who glimpsed for the first time  the power of a triple approach combining differential algebra, Picard-Vessiot extensions and representation theory, in her important work \cite{Previato2008} on the generalization of Burchnall and Chaundy's ideas to ODOs in several variables.  The present work is written in this philosophy, uniting these techniques for the study of coupled spectral problems for MODOs. 
}

\section{Characterizing common solutions}\label{sec-commonSol}

{
In this section we present the construction of the main  tool that allows for our characterization of the existence of common solutions for MODOs. This tool is the differential resultant and we define it under some hypotheses.}
Let us consider matrix differential operators in $\mathcal{R}_{\ell}[D] $,
\begin{equation}\label{def-operators}
P=A_0+A_1 D     \quad \textrm{and} \quad Q=\sum_{j=0}^n B_j D^j , \ \textrm{where } n\geq 1\ ,
\end{equation}
with $A_i$, $B_j$  in the differential ring $\mathcal{R}_{\ell}$. We will assume that the leading  coefficient matrix $A_1$ of $P$ is invertible, allowing for the consideration of a monic operator. 

In this section we will investigate necessary and sufficient conditions on the entries of the coefficient matrices of $P$ and $Q$ for the system 
\begin{equation}\label{eq-system}
    \left\{\begin{matrix}
    PY=\overline{0}\\
    QY=\overline{0}
    \end{matrix}\right. \quad ,\,\,\, Y=(y_1,\ldots ,y_{\ell})^t, \ \overline{0}= (0,\dots ,0)^t ,
\end{equation}
to have a nontrivial solution $\psi = (\psi_1 , \dots ,\psi_{\ell} )^t$, with all the $\psi_i$ in some differential extension $\Sigma$ of the differential field $K$. To guarantee the existence of such an extension, we will associate to system \eqref{eq-system} a differential field, {the classical Picard-Vessiot field extension of $K$ for the system $PY=\overline{0}$, a differential extension of finite algebraic transcendence degree 
containing the entries of a fundamental solution matrix  \cite{VPS}.}

First, observe that the differential equation $PY=\overline{0}$ can be rewritten as
\begin{equation}\label{eq-PV-system}
    DY= N Y \quad \textrm{with } \ N=  -A_1^{-1} A_0 \in \mathcal{R}_{\ell} .
\end{equation}
Next consider the matrix differential recursion in $\mathcal{R}_{\ell}$:
\begin{equation}\label{eq-rec}
    p_0(N):=I_{\ell} \ ,\quad  p_j (N):=p_{j-1}(N)N+(p_{j-1}(N))' \ ,\quad j\geq 1,
\end{equation}
and denote by $M(P,Q)$ the $\ell\times \ell$ matrix in $\mathcal{R}_{\ell}$ defined by
\begin{equation}\label{def-MPQ}
    M(P,Q):=\sum_{j=0}^n B_j p_j(N).
\end{equation}

\begin{remark}\label{rem-Dj}
{Let us consider a Picard-Vessiot extension $\Sigma$} of $K$ for the differential system \eqref{eq-PV-system}. Hence, if $\psi$ is a solution of this system, then $\psi$ is a solution of $PY=\overline{0}$ and reciprocally. Observe that $\Sigma$ is a differential extension of finite algebraic transcendence degree of $K$, see the Appendix.

Given a solution $\psi = (\psi_1 , \dots ,\psi_{\ell} )^t$ of system \eqref{eq-PV-system}, we have $\psi\in \Sigma^{\ell}$ and 
observe that 
\[D^j\psi =p_j(N)\psi\quad, \quad j\geq 1,\]
with $p_j(N)$ defined by \eqref{eq-rec}.
Thus in $\Sigma^{\ell}$ the  derivation is defined by the differential system $DY=NY$.
\end{remark}

From the previous remark the following essential lemma is obtained. It will be used to prove the main results of this paper.

\begin{lemma}\label{lem-essential}
Let $P$ and $Q$ be matrix differential operators as in \eqref{def-operators}. Let $\Psi$ be a fundamental solution matrix for system \eqref{eq-PV-system}. Then,
\begin{equation}\label{eq-QM}
    Q\Psi =  M(P,Q)\Psi.
\end{equation}
\end{lemma}
\begin{proof}
Let us consider a solution $\psi = (\psi_1 , \dots ,\psi_{\ell} )^t$ of system \eqref{eq-PV-system}.
By Remark \ref{rem-Dj}, we have $\psi\in \Sigma^{\ell}$. 
Consequently we obtain
\begin{equation}\label{eq-commonSolution}
    Q\psi = \sum_{j=0}^n B_j D^j\psi=\sum_{j=0}^n B_j p_j(N)\psi=M(P,Q)\psi.
\end{equation}
\end{proof}

\begin{definition}\label{def-defDres}
With notations as above, we define \texttt{the matrix differential resultant} of two matrix differential operators $P$ and $Q$ in $\mathcal{R}_{\ell}$, with $P$ of order one and invertible leading coefficient matrix, to be 
\begin{equation}\label{def-Dres}
\Dres (P,Q)=\det M(P,Q).    
\end{equation}
\end{definition}

We are ready to prove Theorem A from the introduction.
This theorem ensures that the vanishing of the differential resultant of $P$ and $Q$ is a necessary and sufficient condition on the entries of the coefficient matrices of $P$ and $Q$ for the existence of a common nontrivial solution of \eqref{eq-system}, whenever $P$ and $Q$ commute. 

\begin{proof}[Proof of Theorem A]
\begin{enumerate}
    \item Let us assume that there is a nonzero common solution $\psi$ of system \eqref{eq-system}, $\psi\in \Sigma^{\ell}$. By Lemma \ref{lem-essential}, $Q\psi = M(P,Q)\psi$.
Thus the linear map defined by $M(P,Q)$ on $\Sigma^{\ell}$ has a non trivial kernel, and then its determinant is zero.
\item Suppose now that the determinant of the matrix $M(P,Q)$ is zero. Let $\Psi$ be a fundamental matrix of system \eqref{eq-PV-system}, that is an invertible matrix in $M_{\ell} (\Sigma)$, whose columns form a fundamental system of solutions. 
By Lemma \ref{lem-essential} it holds that $Q\Psi =M(P,Q)\Psi$
and therefore the matrix $Q\Psi $ has zero determinant. 

Let us now consider the columns $\{  Q\phi_1, \ldots, Q\phi_{\ell}  \}$ of the matrix $Q\Psi$. They form a system of linearly dependent vectors in $\Sigma^{\ell}$ over the differential field $K$, but also over the field of constants $C$. To verify this we can proceed as follows.

Up to a change of the order of the variables, we can assume that any proper subset of $\{  Q\phi_1, \ldots, Q\phi_r  \}$ is linearly independent over $K$ and 
\begin{equation}\label{eq-Qphi}
    Q\phi_1 = \sum_{i=2}^r c_i Q\phi_i \mbox{ with } c_i \in K.
\end{equation}
{Since $P$ and $Q$ commute then $Q \phi_j$ is also a solution of $P$ and by Remark \ref{rem-Dj} then $D(Q\phi_j )= N (Q\phi_j )$.}
So, differentiating \eqref{eq-Qphi} we get 
\begin{align*}
    \overline{0}&= D(Q\phi_1 )-\sum_{i=2}^r c'_i Q\phi_i - \sum_{i=2}^r c_i D(Q\phi_i )=\\
    &= N(Q\phi_1 )-\sum_{i=2}^r c'_i Q\phi_i - \sum_{i=2}^r c_i N(Q\phi_i ) .
\end{align*}
Then, the equality  $\overline{0}= \sum_{i=2}^r c'_i Q\phi_i$ implies that $c_i' =0$, so they are cons\-tants in $C$.

Finally we have obtained that the vectors $\{  Q\phi_1, \ldots, Q\phi_r  \}$ {are linearly dependent} over $C$. Then, for some constants $c_i \in C$ not all zero, we can consider the vector in $\Sigma^{\ell}$ defined by $\psi:=c_1 \phi_1 + \cdots +c_{\ell} \phi_{\ell}$ verifying
\begin{equation*}
    P\psi = \sum_{i=1}^{\ell} c_i P\phi_i = \overline{0} \quad , \quad 
    Q\psi = \sum_{i=1}^{\ell} c_i Q\phi_i = \overline{0}.
\end{equation*}
So $\psi$ is a common solution, and it is not the null solution because $\Psi$ is a fundamental matrix.
\end{enumerate}
\end{proof}

\begin{corollary}\label{cor-common solutions}
Let $P$ and $Q$ be commuting matrix differential operators as in \eqref{def-operators}. The space of common solutions of \eqref{eq-system} in a Picard-Vessiot extension $\Sigma$ of $K$ for the differential system \eqref{eq-PV-system} is non trivial if and only if $\Dres(P,Q)=0$. Moreover, the subspace of  solutions of   $QY= \overline{0}$ within  the space of solutions of $PY=\overline{0}$ is defined by the equation
\begin{equation}\label{eq-commonSol2}
    M(P,Q)Y=\overline{0}.
\end{equation}
Moreover, its solutions $\psi = (\psi_1 , \dots ,\psi_{\ell} )^t$ have entries belonging to a differential extension of finite algebraic transcendence degree $\Sigma$ of $K$.
\end{corollary}
\begin{proof}
The result follows by Theorem A and Lemma \ref{lem-essential}.
\end{proof}

\begin{remark}\label{remark-1}
{Observe that equation \eqref{eq-commonSol2} together with the structure of $\Sigma$, a Picard-Vessiot field for the system \eqref{eq-PV-system}, determines the space of common solutions of \eqref{eq-system}. }
\end{remark}

The previous results indicate the interest of developing a theory of differential resultants for matrix differential operators in $\mathcal{R}_{\ell}[D]$, which we initiate in this work and plan to develop in the near future. {To define a Sylvester style matrix whose determinant is the resultant, the notion of wronskian for MODOs developed in \cite{Kasman2017} will be useful.} 
In the remaining parts of this paper we will apply Theorem A to guarantee the solvability of matrix-type spectral problems.

\section{Spectral curves for MODOs}\label{sec-spectral curves}

Let us consider algebraic variables $\lambda$ and $\mu$ with respect to $\partial$. Thus $\partial \lambda = 0$ and $\partial \mu = 0$ and we can extend the derivation $\partial$ of $K$ to the polynomial ring $K[\lambda, \mu]$. Hence $(K[\lambda, \mu], \partial)$ is a differential ring whose ring of constants is $(C[\lambda, \mu], \partial)$. Since $ I_{\ell}\lambda$ is a matrix in $M_{\ell}(K[\lambda,\mu])$,
given a matrix differential operator $L$ in $\mathcal{R}_{\ell}[D]$ we will denote by $L-\lambda$ the matrix differential operator $L- I_{\ell}\lambda$ in the matrix ring  $M_{\ell}(K(\lambda,\mu))[D]$, extending the derivation $D$ to matrices with entries in the differential field $K(\lambda,\mu)$.

\medskip

Given differential operators $L$ and $B$ in $\mathcal{R}_{\ell}[D]=M_{\ell}(K)[D] $, 
we consider the spectral problem
\begin{equation}\label{eq-sp-system}
    LY=\lambda  Y\quad , \quad 
    BY=\mu Y
    \quad ,\,\,\, Y=(y_1,\ldots ,y_{\ell})^t .
\end{equation}
{
The spectral problem \eqref{eq-sp-system} can be studied using Theorem A for
\begin{equation}\label{def-operatorsLB}
L=A_0+A_1 D     \quad \textrm{and} \quad B=\sum_{j=0}^n B_j D^j , \ \textrm{where } n\geq 1 \ ,
\end{equation}
assuming that $A_1$ is invertible.
More precisely, we will apply Theorem A to the matrix differential operators}
\begin{equation}
    P=L-\lambda = (A_0 - I_{\ell} \lambda) + A_1 D , \quad     Q=B-\mu= (B_0 - I_{\ell} \mu) +\sum_{j=1}^n B_j D^j.
\end{equation}

Let $N_{\lambda}$ be the matrix  $N_{\lambda}=-A_1^{-1}(A_0-I_{\ell}\lambda )$ in $M_{\ell} (K(\lambda , \mu ))$, and consider  the differential system
\begin{equation}\label{eq-PV-system_lambda}
    DY= N_\lambda Y \quad \textrm{ with } \ N_{\lambda}=-A_1^{-1}(A_0- I_{\ell}\lambda ) .
\end{equation}
We will consider next the matrix differential resultant  defined in \eqref{def-Dres} for $P=L-\lambda$ and $Q=B-\mu$,
\begin{equation}\label{eq-dresLB}
    \Dres(L-\lambda,B-\mu)=\det M(L-\lambda, B-\mu)
\end{equation}
where
\begin{equation}\label{eq-matrix}
    M(L-\lambda, B-\mu)=B_0-I_{\ell} \mu +\sum_{j=1}^n B_j p_j(N_{\lambda}),
\end{equation}
with  $p_j$ defined by the recursion \eqref{eq-rec}. With the above assumptions we have the following statements.

\begin{lemma}\label{lemma-1}
The differential resultant $\Dres(L-\lambda,B-\mu)$ is a polynomial $f(\lambda,\mu)$ in $K[\lambda,\mu]$, of degrees $\ell$ in $\mu$ and less than or equal to $\ell n$ in $\lambda$. More precisely
\begin{equation}\label{eq-degreesf}
{f(\lambda,\mu)=(-1)^{\ell} \mu^{\ell}+\det(B_n)\det(A_1^{-1})^n\lambda^{n\ell}+q(\lambda,\mu),}
\end{equation}
where the degree of $q$ is less than $\ell$ in $\mu$ and less than $n\ell$ in $\lambda$.
{Thus the degree in $\lambda$ is exactly $n\ell$ if and only if $\det(B_n)\neq 0$.}
\end{lemma}
\begin{proof}
By \eqref{eq-matrix} $M(L-\lambda, B-\mu)$
is an $\ell\times \ell$ matrix with entries in $K[\lambda,\mu]$.
Thus $f(\lambda,\mu)=\det M(L-\lambda, B-\mu)$ is a polynomial in  $K[\lambda,\mu]$. {Let us use \eqref{eq-matrix} to obtain \eqref{eq-degreesf}.} 
\begin{enumerate}
    \item We will prove by induction the following claim: 
    \begin{equation}\label{eq-indpj}
        p_j(N_{\lambda})=(A_1^{-1})^j\lambda^j+\Delta_{j-1},\mbox{ for every }j\geq 1
    \end{equation}
    where $\Delta_{j-1}$ is a matrix whose entries are polynomials in $\lambda$ of degree less than or equal to $j-1$.

    By \eqref{eq-rec} we know that $p_1(N_{\lambda})=N_{\lambda}=A_1^{-1}\lambda+\Delta_0$ with $\Delta_0=-A_1^{-1}A_0$. Let us assume \eqref{eq-indpj} to compute 
    \begin{align*}
        p_{j+1}(N_{\lambda})&=p_{j}(N_{\lambda}) N_{\lambda}+p_{j}(N_{\lambda})'=\\
        &=((A_1^{-1})^j\lambda^j+\Delta_{j-1})(A_1^{-1}\lambda+\Delta_0)+p_{j}(N_{\lambda})'=\\
        &=(A_1^{-1})^{j+1}\lambda^{j+1}+\Delta_{j},
    \end{align*}
    where
    \begin{align*}
        \Delta_{j}=(A_1^{-1})^j\Delta_0\lambda^j+\Delta_{j-1}A_1^{-1}\lambda+\Delta_{j-1}\Delta_0+p_{j}(N_{\lambda})'
    \end{align*}
    is a matrix whose entries are polynomials in $\lambda$ of degree less than or equal to $j$.

    \item We can now prove \eqref{eq-degreesf}. The $\ell\times \ell$ matrix
   \begin{align*}
    M=M(L-\lambda, B-\mu)&=B_0-I_{\ell} \mu +\sum_{j=1}^n B_j p_j(N_{\lambda})=\\
    &=B_0-I_{\ell} \mu +\sum_{j=1}^n B_j (A_1^{-1})^j \lambda^j +B_j\Delta_j
    \end{align*}
has entries $\alpha_{i,j}\lambda^n+\beta_{i,j}(\lambda)-\delta_{ij} \,\mu$ \ where  $\delta_{ij}$ is the Kronecker delta, 
\[
B_n(A_1^{-1})^n=(\alpha_{i,j})\in  M_{\ell}(K)
\]
and 
\[B_0 +\sum_{j=1}^{n-1} B_j p_j(N_{\lambda})=(\beta_{i,j}(\lambda))\in M_{\ell}(K[\lambda]).\]
It is now immediate that \eqref{eq-degreesf} follows from
\begin{align*}
    \det(M)&=\det(\alpha_{i,j}\lambda^n+\beta_{i,j}(\lambda)-\delta_{ij} \, \mu)=(-\mu)^{\ell}+\det(\alpha_{i,j})\lambda^{n\ell}+q(\lambda,\mu).
\end{align*}
\end{enumerate}
\end{proof}
\begin{remark}\label{rem-PVE}\label{remark-PVE}
We can consider the operators $P=L-\lambda$ and $Q=B-\mu$ with matrix coefficients with entries in $\mathcal
{F}=K(\lambda,\mu)$.
Let  $\overline{\mathcal
{F}}$  be an algebraic closure of the differential field $\mathcal
{F}$, and $\mathcal{C}$ its field of constants, which is known to be algebraically closed, see \eqref{eq-constants-cl-appenddix} in the Appendix.  
In consequence, by Remark \ref{rem-E-appendix} in the Appendix, there exists a fundamental matrix $\Psi_\lambda$ of \eqref{eq-PV-system_lambda} in $M_{\ell}(\mathcal{E})$  such that {$D\Psi_\lambda=N_{\lambda}\Psi_\lambda$,  or equivalently}
\begin{equation}\label{eq-with-lamda}
L\Psi_\lambda = \lambda \Psi_\lambda    ,
\end{equation}
where $\mathcal{E}$ is a Picard-Vessiot extension of $\overline{\mathcal
{F}}$ for this differential system. 

To be more precise, the matrix coefficients of $P$ and $Q$ have entries in the differential ring $K[\lambda,\mu]$, whose ring of constants is $C[\lambda,\mu]$. Observe that $C[\lambda,\mu]\subset \mathcal{C}$ and $K[\lambda,\mu]\cap \mathcal{C}=C[\lambda,\mu]$. For details, see Lemma \ref{lemma-para-Previato}  in the Appendix.
\end{remark}

The next lemmas will be necessary to prove the main result in this section. 

\begin{lemma}\label{lem-V}
Given a solution $ \psi$  of \eqref{eq-PV-system_lambda} in  $\mathcal{E}^{\ell}$, the next identity holds
\begin{equation}
    (B-\mu)(\psi)=M(L-\lambda , B-\mu )\cdot \psi.
\end{equation}
\end{lemma}
\begin{proof}
By Remark \ref{remark-PVE} and Lemma \ref{lem-essential} for $P=L-\lambda$ and $Q=B-\mu$ the equality follows.
\end{proof}

\begin{lemma}\label{lem-Delta0}
Let us assume that $L$ and $B$ commute.
Given a fundamental matrix of solutions  $\Psi_{\lambda}$ of $LY=\lambda Y$ in $M_\ell (\mathcal{E})$ the next identity holds
\begin{equation}
    (B-\mu)(\Psi_\lambda) =\Psi_\lambda \cdot \Delta \ ,
\end{equation}
for some $\ell\times \ell$ matrix $\Delta$ with entries in the field of constants $\mathcal{C}$ of $\mathcal{E}$. 
\end{lemma}
\begin{proof}
With  notations as above, we consider $\Psi_\lambda$ a fundamental matrix sa\-tisfying \eqref{eq-with-lamda}. Morever, observe that $(B-\mu )\Psi_\lambda$ satisfies  \eqref{eq-with-lamda} as well, since 
\begin{equation*}
    (L-\lambda)((B-\mu)\Psi_\lambda )=(B-\mu) ((L-\lambda)\Psi_\lambda )=0.
\end{equation*}
Hence, $\Delta=\Psi_\lambda^{-1} \cdot (B-\mu)\Psi_\lambda$
is differentially constant.
\end{proof}

We are ready to prove Previato's Theorem for MODOs in this situation, see \cite{Prev} and \cite{MRZ1} for Previato's Theorem for ODOs.

\begin{proof}[Proof of Theorem B]
With the previous notations, we consider $\Psi_\lambda$ a fundamental matrix sa\-tisfying \eqref{eq-with-lamda}. By Lemma \ref{lem-Delta0}
\begin{equation}
    (B-\mu)\Psi_\lambda =\Psi_\lambda \cdot \Delta \ ,
\end{equation}
for some matrix $\Delta$ with entries in $\mathcal{C}$. On the other hand, by Lemma \ref{lem-V}
\begin{equation*}
    (B-\mu)\Psi_{\lambda}=M(L-\lambda , B-\mu )\Psi_{\lambda}  .
\end{equation*}
Therefore, since the fundamental matrix $\Psi_\lambda$ is an invertible matrix, we obtain that $\det (M(L-\lambda , B-\mu )= \det (\Delta )$,  so $\Dres(L-\lambda,B-\mu)$ is a polynomial in $C [\lambda ,\mu ]$, since $K[\lambda,\mu]\cap \mathcal{C}=C[\lambda,\mu]$.
\end{proof}

From now on in this paper we will assume that $L$ and $B$  commute. 
Therefore, using the matrix differential resultant we compute 
\begin{equation}\label{def-resultant}
    f(\lambda,\mu)=\Dres(L-\lambda,B-\mu),
\end{equation}
the defining polynomial of the affine plane algebraic curve 
\begin{equation}\label{eq-scurve}
\Gamma=\{(\lambda,\mu)\in C^2\mid f(\lambda,\mu)=0\}.  
\end{equation}

We will call $\Gamma$ \texttt{the spectral curve} of the pair $L,B$. In sections \ref{sec-BC} and \ref{sec-Classification} we establish an isomorphism between the ring of this curve and the commutative subring $C[L,B]$ of the ring of MODOs, under an appropriate hypothesis on $\Gamma$.

In general $\Gamma$ is not an irreducible curve, that is, $\Gamma$ may have more than one irreducible component. In Section \ref{sec-AKNS}, we give Example \ref{ex-irreducible} with an irreducible spectral curve and Example \ref{ex-nonirred} where the spectral curve has two irreducible components. Thus one open issue is to characterize the irreducibility of the curve $\Gamma$ or study the role of the irreducible components of $\Gamma$ in the study of the direct and inverse spectral problems.

\medskip

We finish this section applying Theorem A to the pair $L-\lambda_0$ and $B-\mu_0$, for some arbitrary point $P=(\lambda_0 ,\mu_0 )$ in $C^2$. We conclude that the spectral problem \begin{equation}
    L Y =\lambda_0 Y \quad , \quad B Y = \mu_0 Y \ .
\end{equation} is a coupled problem  by the algebraic relation $f(\lambda_0 , \mu_0 ) =0$. 

\begin{corollary}\label{cor-coupled}
Given commuting matrix differential operators $L$ and $B$ in $R_{\ell}$[D], with $L=A_0+A_1 D$  and invertible leading coefficient, let us consider the polynomial $f(\lambda,\mu)$ defined by $\Dres(L-\lambda,B-\mu)$. Let $P=(\lambda_0 ,\mu_0) \in C^2$. 
The spectral problem 
\begin{equation}
    L Y =\lambda_0 Y \quad , \quad B Y = \mu_0 Y \ .
\end{equation}
 has a nontrivial solution  if and only if $f(P)=0$, that is $P$ is a point on the spectral curve $\Gamma$ defined in \eqref{eq-scurve}. Moreover the common solution $\psi$ belongs to ${\Sigma_0 }^{\ell}$, where $\Sigma_0$ is a Picard-Vessiot extension for the linear differential system 
 \begin{equation}
    DY= N_{\lambda_0 }Y \quad \textrm{ with } \ N_{\lambda_0 }=-A_1^{-1}(A_0- {\lambda_0 } I_{\ell} ) .
\end{equation}
\end{corollary}

\section{{The ideal of Burchnall-Chaundy polynomials}}\label{sec-BC}

Let us consider commuting matrix differential operators $L$ and $B$ in $\mathcal{R}_{\ell}[D]$, of respective orders $1$ and $n\geq 1$. We  assume that the leading coefficient of $L$ is an invertible matrix.
Using the matrix differential resultant we compute 
\begin{equation}\label{def-resultant}
    f(\lambda,\mu)=\Dres(L-\lambda,B-\mu),
\end{equation}
the defining polynomial of the spectral curve $\Gamma$ of the pair $L,B$, see \eqref{eq-scurve}.
We will prove next that $f$ is a good candidate as a polynomial of Burchnall-Chaundy type for the pair of commuting MODOs $L,B$. We propose below the next conjecture: when $\lambda$ is replaced by $L$ and $\mu$ is replaced by $B$ then $f$ becomes the zero matrix differential operator, see Conjecture \ref{conjecture}.

\medskip

Consider the natural ring homomorphism  
\begin{equation}\label{eq-rho}
    \rho:C[\lambda,\mu]\longrightarrow \mathcal{R}_{\ell}[D],
\end{equation}
defined by $\rho (c)=c I_{\ell}$, for every $ c\in C$,
\begin{equation}\label{eq-morphism1}
    \lambda\mapsto L\mbox{ and }\mu\mapsto B .
\end{equation}
Thus
\[\sum a_{i,j} \lambda^i\mu^j \mapsto \sum a_{i,j} L^iB^j.\]
Observe that the image of $\rho$ is a commutative subalgebra of $\mathcal{R}_{\ell}[D]$, namely the commutative algebra
\begin{equation}
C[L,B]:=\left\{\sum a_{i,j} L^iB^j\mid a_{i,j}\in C\right\}.
\end{equation}
Given $g\in C[\lambda,\mu]$ we will denote its image $\rho(g)$ by $g(L,B)$.

\medskip

Let us consider the kernel $\Ker (\rho)$ of $\rho$ and 
observe that the polynomials in $\Ker (\rho)$ play the role of Burchnall-Chaundy polynomials in the case of ODOs \cite{BC1}. 

\begin{definition}
Given a pair of commuting MODOs $L$ and $B$ in $\mathcal{R}_{\ell}[D]$, with $L$ of order $1$ and invertible leading coefficient matrix,
we will say that $g\in C[\lambda,\mu]$ is \texttt{a Burchnall-Chaundy (BC) polynomial} of the pair $L,B$ if $$g(L,B)=\m0.$$
\end{definition}

We will prove next that $f(L,B)$, for $f$ defined in \eqref{def-resultant}, is a MODO, that becomes zero when considering its action on solutions of $L-\lambda$.

\begin{theorem}\label{thm-ideal1}
Given commuting MODOs $L$ and $B$ in $\mathcal{R}_{\ell}[D]$, we  assume that $L$ has order $1$, with invertible leading coefficient. 
Let us consider the polynomial $f(\lambda,\mu)=\Dres(L-\lambda,B-\mu)$ in $C[\lambda,\mu]$. 
Then $f(L,B)(\Psi_{\lambda})=\m0$, for any fundamental matrix $\Psi_{\lambda}$ of the system $LY=\lambda Y$.
\end{theorem}
\begin{proof}
Let us prove that $f$ belongs to $\Ker(\rho)$.
We denote by $V_\lambda$ the eigenspace of $L$ corresponding to the eigenvalue $\lambda$,
{
\begin{equation}
    V_{\lambda}=\{\psi\in \mathcal{E}^{\ell}\mid D\psi=N_{\lambda}\psi\}.
\end{equation}
}
{By Remark \ref{rem-PVE} and Picard-Vessiot theory \cite{VPS}, $V_\lambda$ is an $\ell$-dimensional $\mathcal{C}$-vector space}. We define $B_\lambda$ as the restriction of $B$ to the eigenspace $V_\lambda$:
\begin{equation}
  B_\lambda : V_\lambda \rightarrow V_\lambda \quad , \quad   B_\lambda (\psi) = B(\psi) = (B-\mu )(\psi) +\mu \psi = (M+\mu I_{\ell} )\psi ,
\end{equation}
where $M=M(L-\lambda , B-\mu)$ by {Lemma \ref{lem-V}}. In consequence, the matrix  differential  resultant is the characteristic polynomial of $B_\lambda$ since the following equality holds
\begin{equation}
    \det (B_\lambda -\mu I_{\ell} ) = \det ( M ) = f(\lambda ,\mu ) .
\end{equation}
By the Cayley-Hamilton Theorem, $f(\lambda , B)=\m0 $. In consequence for $\Psi_\lambda $ a fundamental matrix of \eqref{eq-PV-system} the  Burchnall–Chaundy type equality holds:
\begin{equation}\label{eq-CH}
\m0=  f(\lambda , B)(\Psi_\lambda) =  f(L , B)(\Psi_\lambda).
\end{equation}
\end{proof}

\begin{remark}\label{rem-BCpol}
The conclusion of Theorem \ref{thm-ideal1} can be improved in some cases. In \cite{Grinevich1987}, Grinevich studied commuting MODOs $L_1$ and $L_2$, of orders $m$ and $n$ respectively, with analytic coefficients, under some assumptions on their leading terms. In particular, for $L_1=L$, with $m=1$ and $L_2=B$ and matrix coefficients in $M_{\ell} (\mathbb{C}\{x\})$, with entries in the ring $\mathbb{C}\{x\}$ of convergent power series, \cite{Grinevich1987}, Lemma 1, states that $f(L,B)$ is the zero operator. In this case, observe that our approach allows the computation of a defining polynomial $f$ of the spectral curve.

In Section \ref{sec-AKNS}, we show that $f(L,B)$ is the zero MODO, that is $f$ is a $\BC$-polynomial, for the first case of the AKNS hierarchy. 
\end{remark}

The previous remark motivates the following conjecture, the proof of which is a challenging open problem.

\medskip
\begin{conjecture}\label{conjecture}
\noindent  Let $K$ be a differential field, whose field of constants $C$ is algebraically closed and of zero characteristic. Let us consider commuting MODOs $L$ and $B$ with coefficients in $M_{\ell}(K)$, of respective orders $1$ and $n\geq 1$ as in \eqref{def-operatorsLB}.
Then the differential resultant \eqref{def-resultant} 
is a $\BC$-polynomial, that is $f(L,B)=\m0$. 
\end{conjecture}

\medskip

\begin{definition}
With notations as above, we define the \texttt{Burchnall-Chaundy (BC) ideal} of the pair $L,B$ to be the ideal in $C[\lambda ,\mu ]$ defined by
\begin{equation}
    \BC (L,B) =Ker (\rho) =\{g\in C[\lambda,\mu]\mid g(L,B)=\m0\}.
\end{equation}
\end{definition}

\medskip

Under the assumption $f(L,B)=\m0$ then $f$ is a $BC$ polynomial and $\BC (L,B)$ is a nonzero ideal.
From the definition of $\rho$ in \eqref{eq-rho} we can now consider the isomorphism 
\begin{equation*}\label{eq-iso1}
    \frac{{C[\lambda,\mu]}}{\BC (L,B)} \simeq C[L,B].
\end{equation*}
A natural question is whether $C[L,B]$ is isomorphic to the ring of the curve $\Gamma$ defined in \eqref{eq-scurve}, but one encounters some difficulties related with the algebraic nature of the situation described.
More precisely, since we are working with matrices, $C[L,B]$ is not in general an integral domain. In fact, from the previous isomorphism, we know that $C[L,B]$ is an integral domain if and only if $\BC (L,B)$ is a prime ideal.
The next technical lemma will allow us to obtain some conclusions about this issue.

\begin{lemma}\label{lem-Delta}
Let us assume that $L$ and $B$ commute, and let us consider a fundamental matrix of solutions  $\Psi_{\lambda}$ of $LY=\lambda Y$ in $M_\ell (\mathcal{E})$. Given a polynomial $g(\lambda,\mu)=\sum_{i=0}^p a_i(\lambda) \mu^i$ in $C[\lambda, \mu]$ the next identity holds
\begin{equation}\label{eq-gDelta}
   g(\lambda,\mu)\Psi_{\lambda}= g(L,B)(\Psi_{\lambda})+\Theta(g)\Psi_{\lambda} \Delta,
\end{equation}
for some matrix $\Theta(g)$ with entries in $K[\lambda ,\mu ]$ and some matrix $\Delta$ with entries in $\mathcal{C}$, whose determinant equals $\Dres(L-\lambda,B-\mu)$.
\end{lemma}
\begin{proof}
By Lemma \ref{lem-Delta0}, we can define the constant matrix
\[\Delta:=\Psi_{\lambda}^{-1}\cdot (B-\mu)(\Psi_{\lambda})\in M_{\ell} (\mathcal{C}).\]
Let us define the $\mathcal{C}$-linear operator
\[\hat{\Delta}:M_{\ell}(\mathcal{E})\rightarrow M_{\ell}(\mathcal{E}), \mbox{ by }\hat{\Delta}(S)=S\cdot \Delta,\]
on the $\mathcal{C}$-vector space $M_{\ell}(\mathcal{E})$. Thus $\Psi_{\lambda}$ is a $\mu$ eigenvector for $B-\hat{\Delta}$ and it follows immediately that
\begin{equation}\label{eq-mu-p}
     ( B-\hat{\Delta} )^i (\Psi_{\lambda})=\mu^i \Psi_{\lambda},\,\,\,i\geq 1.
\end{equation}
In addition observe that the $\mathcal{C}$-linear operators $B$ and $\hat{\Delta}$ commute, since $\Delta$ has constant entries. The following formula is a consequence of this fact,  
\begin{equation}\label{eq-BDelta}
( B-\hat{\Delta} )^i (\Psi_{\lambda}) := \sum_{k=0}^i (-1)^{i-k}\binom{i}{k}B^k (\Psi_{\lambda})\cdot \Delta^{i-k},\,\,\,i\geq 1. 
\end{equation}
Given a matrix $A$ of size $\ell\times\ell$ we assume that $A^0=I_{\ell}$.

Now, consider an arbitrary polynomial of degree $p$ in $\mu$,\break $g(\lambda,\mu)=\sum_{i=0}^p a_i(\lambda) \mu^i$ in $C[\lambda, \mu]$. 
Then, applying \eqref{eq-mu-p} we obtain
\begin{equation}\label{eq-g}
    \begin{array}{ll}
   g(\lambda,\mu)  \Psi_{\lambda} &= \displaystyle \sum_{i=0}^p a_i(\lambda) \mu^i \Psi_{\lambda}=\sum_{i=0}^p a_i(\lambda) (B-\hat{\Delta})^i(\Psi_{\lambda}) =\\
           &=\displaystyle g(\lambda,B)(\Psi_{\lambda})
           +\sum_{i=1}^p a_i(\lambda) [(B-\hat{\Delta})^i-B^i](\Psi_{\lambda}).
    \end{array}
\end{equation}

Let $M=M(L-\lambda, B-\mu)$ and $M_{\lambda}=M+I_{\ell}\mu$.
By Lemmas \ref{lem-V} and \ref{lem-Delta0} we have 
\begin{equation}\label{eq-M}
    M\cdot \Psi_{\lambda}= \Psi_{\lambda}\cdot\Delta\mbox{ and }M_{\lambda}\cdot \Psi_{\lambda}= \Psi_{\lambda}\cdot(\Delta+\mu I_{\ell}).
\end{equation}
Since $B(\Psi_{\lambda})=M_{\lambda}\cdot \Psi_{\lambda}$ these by induction on $k$ implies
\[B^k(\Psi_{\lambda})=M_{\lambda}^k\cdot \Psi_{\lambda},\,\,\,k\geq 1.\]

By \eqref{eq-BDelta}, for $1\leq i\leq p$
\begin{align*}
    &[(B-\hat{\Delta})^i-B^i](\Psi_{\lambda})=
    \left(\sum_{k=0}^{i-1} (-1)^{i-k}\binom{i}{k}B^k (\Psi_{\lambda})\cdot\Delta^{i-1-k}\right)\cdot\Delta= \Theta_i\cdot\Psi_{\lambda}\cdot\Delta,
\end{align*}

where, by \eqref{eq-M} we have matrices
\[\Theta_i= \sum_{k=0}^{i-1} (-1)^{i-k}\binom{i}{k}M_{\lambda}^k M^{i-1-k}\in M_{\ell}(K[\lambda,\mu]).\]

Finally, from \eqref{eq-g} and equality $L\Psi_{\lambda} = \lambda \Psi_{\lambda}$ the result follows
\begin{equation}
g(\lambda,\mu)  \Psi_{\lambda} = g(L ,B )(\Psi_{\lambda}) +\Theta(g)\cdot \Psi_{\lambda}\cdot\Delta.
 \end{equation}
 where $\Theta(g)=\sum_{i=1}^p a_i(\lambda) \Theta_i\in M_{\ell}(K[\lambda,\mu])$.

\end{proof}

\begin{theorem}\label{thm-ideal}
Let us consider commuting MODOs $L$ and $B$ in $\mathcal{R}_{\ell}[D]$, with $L$ of order one and invertible leading coefficient.
Given the polynomial $f(\lambda,\mu)=\Dres(L-\lambda,B-\mu)$ in $C[\lambda,\mu]$, then the following inclusions of ideals in $C[\lambda,\mu]$ hold
\begin{equation}
    \BC(L,B)\subseteq (f_{red}),
\end{equation}
where $f_{red} = h_1 \cdots h_s$, with  $h_1 , \dots ,h_s$ the distinct irreducible
factors of $f$. Moreover, if $f(L,B)=\m0$ then 
\begin{equation}
    (f)\subseteq\BC(L,B)\subseteq (f_{red}).
\end{equation}
\end{theorem}
\begin{proof}
Given $g\in  \BC(L,B)=\Ker(\rho)$, by Lemma \ref{lem-Delta}, 
\begin{equation}
    g(\lambda,\mu)\Psi_{\lambda}=\Theta(g)\Psi_{\lambda}\Delta.
\end{equation}
Taking determinants on both sides we obtain $g^{\ell}= \det(\Theta(g))f(\lambda,\mu)$, since $\det(\Delta)=f(\lambda,\mu)$. It follows that $\det(\Theta(g))\in C[\lambda,\mu]$. This implies that the square free part $f_{red}$ of $f$ divides the square free part of  $g$ and proves that $\Ker(\rho)\subset (f_{red})$. 
In addition, if we assume $f(L,B)=\m0$, then $(f)\subseteq \BC(L,B)$.
\end{proof}

Since the polynomial $f(\lambda,\mu)=\Dres(L-\lambda,B-\mu)$ may not be irreducible, see Example \ref{ex-nonirred} in Section \ref{sec-AKNS}, the ideal $(f)$ may not be a prime ideal. The next corollary presents the conclusions in a particular situation.

\begin{corollary}\label{cor-trivial}
Let us consider commuting MODOs $L$ and $B$ in $\mathcal{R}_{\ell}[D]$, with $L$ of order one and invertible leading coefficient.
Given $f(\lambda,\mu)=\Dres(L-\lambda,B-\mu)$, then the following statements hold for some polynomial $h=\mu-R(\lambda)$ in  $C[\lambda,\mu]$:
\begin{enumerate}
    \item If $B\in C[L]$ then $\BC(L,B)=(h)$ and $f(\lambda,\mu)=h^{\ell}$.
    
    \item If $f=h^{\ell}$ and $f(L,B)=\m0$ then $\BC(L,B)=(h^r)$, for $r$ minimal such that $h(L,B)^r=0$, $1\leq r\leq \ell$. In particular, if $r=1$ then $B\in C[L]$.
\end{enumerate}
\end{corollary}
\begin{proof}
Recall that by Lemma \ref{lemma-1} the polynomial $f$ has degree $\ell$ in $\mu$.
Observe that if $B\in C[L]$ then $B-R(L)=0$, for some $R(\lambda)\in C[\lambda]$. Thus $h=\mu-R(\lambda)$ belongs to $\BC(L,B)$ and by Theorem  \ref{thm-ideal} then $f_{red}=h$. Thus $f=h^{\ell}$. Similarly we can prove 2.
\end{proof}

\section{{Commutative algebras of MODOs }}\label{sec-Classification}

As in the previous sections, let us consider $L$ in $\mathcal{R}_{\ell}[D]$ of order one and invertible leading coefficient matrix. Given $B$ in the centralizer of $\mathcal{C}(L)$ of $L$ in $\mathcal{R}_{\ell}[D]$, let us assume that $B$ is non trivial, that is $B\notin C[L]$, see Corollary \ref{cor-trivial}. From now on, we will assume that $f(L,B)=\m0$. {In Remark \ref{rem-BCpol} we provided evidences showing that this hypothesis is satisfied in important families of examples.}

\medskip

In this section we give an algorithm to describe the commutative algebras $C[L,B]$ of MODOs in terms of products of rings of irreducible algebraic curves. We start with the case where $\Gamma$ is an irreducible plane algebraic curve, that is, its defining polynomial $f$ has a unique irreducible factor $h$, and then $f=h^\sigma$. 

{\begin{theorem}\label{thm-isocurve}
Let us consider commuting MODOs $L$ and $B$ in $\mathcal{R}_{\ell}[D]$, with $L$ of order one and invertible leading coefficient.
Let us assume that $\Gamma$ is an irreducible plane algebraic curve defined by $f=h^{\sigma}$ as in \eqref{def-resultant} and that $f(L,B)=\m0$. Then $ \BC (L,B)=(h^r)$
for a positive integer $r$ minimal such that $h(L,B)^r=\m0$, $1\leq r\leq \sigma$.
Moreover, 
\begin{equation}
     C[L,B]\simeq \frac{{C[\lambda,\mu]}}{(h^r)} .
\end{equation}
If in addition $f$ is irreducible in $C[\lambda,\mu]$, the algebra $C[L,B]$ is an integral domain whose maximal spectrum is isomorphic to $\Gamma$.
\end{theorem}
}
\begin{proof}
By Theorem \ref{thm-ideal} we know that  $(h^{\sigma} )\subset\BC (L,B)\subset (h)$. Then there exists $r$ minimal verifying $1\leq r\leq \sigma$ and $h(L,B)^r=0$. Thus $\BC (L,B)=(h^r)$. Therefore, if $f$ is irreducible, then $r=1$, and the result follows.
\end{proof}

\medskip

We conclude that if $\Gamma$ is an irreducible and reduced plane algebraic curve, then the commutative ring of MODOs $C[L,B]$ is isomorphic to the ring of regular functions on $\Gamma$ and therefore is an integral domain.
In addition Theorem \ref{thm-isocurve} indicates that if $f$ is not irreducible, we can check whether its irreducible square free part is a BC polynomial, and otherwise we can find a minimal power $h^r$ which is a BC polynomial for the pair $L,B$.

\medskip

The previous result allows to complete the classification of commutative algebras of MODOs $C[L,B]$ for order one operators $L$ with invertible leading coefficient matrix when $K$ is any differential field with algebraically closed  field of constants $C$, in the zero characteristic case. See Mulase et al. \cite{MuKi} for  the differential field $K=C((x))$.

\medskip

For a non irreducible curve $\Gamma$, whose defining polynomial is $f=h_1^{\sigma_1 } \cdots h_s^{\sigma_s}$, where each $h_i$ is irreducible, we consider now each irreducible component 
\begin{equation}
\Gamma_i=\{(\lambda,\mu)\in C^2\mid h_i(\lambda,\mu)=0\}.
\end{equation}
We repeat the argument of Theorem \ref{thm-isocurve} for each irreducible component $\Gamma_i$ of the curve $\Gamma$. 

\medskip

{\bf Algorithm \texttt{BC-generator}.}
{\it \texttt{Given} commuting MODOs $L$ and $B$ in $\mathcal{R}_{\ell}[D]$, with $L$ of order one and invertible leading coefficient,
\texttt{return} a polynomial $F$ in $C[\lambda,\mu]$ such that $\BC(L,B)=(F)$.
\begin{enumerate}
    \item Compute the differential resultant $f(\lambda,\mu)=\Dres(L-\lambda ,B-\mu)$.
    
    \item If $f(L,B)=\m0$ then factor $f$ to obtain $h_1^{\sigma_1 } \cdots h_s^{\sigma_s }$, each $h_i$ irreducible in $C[\lambda,\mu]$.
   
    \item 
    For each $i=1,\ldots ,s$, compute the minimal integer $r_i$, with $1\leq r_i\leq \sigma_i$, such that 
    $$\prod_i h_i(L,B)^{r_i}=\m0.$$
    
    \item \texttt{Return} $F=h_1^{r_1 } \cdots h_s^{r_s}$.
\end{enumerate}
}

We prove the correctness of Algorithm \texttt{BC-generator} in the next proof of Theorem C. Observe that by Remark \ref{rem-BCpol} the condition $f(L,B)=\m0$ is always satisfies for matrices whose entries are analytic functions. 

\begin{proof}[Proof of Theorem C]
By Theorem \ref{thm-ideal}, we have  $(f)\subset\BC (L,B)\subset (h_1 \cdots h_s)$. Since $f(L,B)=\m0$, there exist minimal integers $r_i$, with $1\leq r_i\leq \sigma_i$ such that 
$$\prod_i h_i(L,B)^{r_i}=\m0.$$
Thus $F=h_1^{r_1}\cdots h_s^{r_s}\in \BC (L,B)$ and $\BC (L,B)=(F)$.

We can establish the isomorphism in \eqref{eq-decomposition}, using the classical decomposition of the quotient
\[C[L,B]\simeq \frac{C[\lambda , \mu ] }{\BC (L,B)}=\frac{C[\lambda , \mu ] }{(h_1^{r_1 } \cdots h_s^{r_s})},\]
see for instance  \cite{AM} pag. 7. 
\end{proof}

\section{The AKNS system }\label{sec-AKNS}

In 1974, Ablowitz, Kaup, Newell and Segur introduced in \cite{AKNS1974}, a system of integrable nonlinear evolutionary equations called the AKNS system, whose stationary version is the nonlinear differential system
\begin{equation}\label{eq-AKNS}
  \begin{array}{cc}
     \frac{\imath}{2} v_{xx}+\imath v^2 u &  =0 \ ,\\
       \frac{\imath}{2} u_{xx}+\imath v u^2 &  =0 \ ,
  \end{array}  
\end{equation}
which can be regarded as the  complexified nonlinear stationary  Sch\"odinger (nS) equation
\begin{equation}
  -\frac{\imath}{2}  u_{xx}\pm \imath |u|^2 u =0 ,
\end{equation}
for $v=\mp{u}^\ast$ and  $u^\ast$ the complex conjugate of $u$. In \cite{GH}, Gesztesy and  Holden provide a matrix recursion for an integrable matrix hierarchy, called the AKNS hierarchy, whose first non trivial member has  equations \eqref{eq-AKNS} as integrability conditions. Its matrix presentation provides a pair of MODOs $L$, $B$ in $\mathcal{R}_{2}[D]$ for matrix coefficients with entries in the differential field $K=\mathbb{C} \langle u, v\rangle$,  where $u$ and $v$ satisfy \eqref{eq-AKNS}, see  examples \ref{ex-irreducible} and \ref{ex-nonirred} of this section.
Specifically, let $L$ be the $2\times 2$ matrix
\begin{equation}\label{mat-1}
L=\imath\left[\begin{matrix}D &u\\v&-D\end{matrix}\right]   = A_0 +A_1 D, \textrm{with } 
A_0 =\imath\left[\begin{matrix}0&u\\v&0\end{matrix}\right] ,
A_1 = \imath\left[\begin{matrix}1&0\\0&-1\end{matrix}\right]
\ .
\end{equation}
Next consider the second order matrix differential operator
\begin{equation}\label{mat-2}
    B= \imath\left[\begin{matrix}-2D^2-uv&-2uD-u_x\\-2vD-v_x&2D^2+uv\end{matrix}\right]
    =
    B_0 +
    B_1 D +B_2 D^2 \ ,
\end{equation}
where
\begin{equation*}
    B_0 =\imath\left[\begin{matrix}-uv&-u_x\\-v_x&uv\end{matrix}\right]
    \ , \ 
    B_1 =
    \imath\left[\begin{matrix}0&-2u\\-2v&0\end{matrix}\right]
    \ , \ 
    B_2 =\imath\left[\begin{matrix}-2&0\\0&2\end{matrix}\right]
    \ .
\end{equation*}
This matrix $B$ is matrix $Q_2$ in \cite{GH}, p. $180$, for integration constants $c_1 =0$ and $c_2 =0$ and potentials $p=v$ and $q=-u$. The matrices \eqref{mat-1} and \eqref{mat-2}  can be found in \cite{previato} as $L_1$ and $L_2$ respectively, and in that context, $u$ and $v$ are solutions to a complexified non-linear Schr\"odinger (NLS) system, 
so that below we will take $v$ to be the complex conjugate of $u$, denoted by $u^\ast$,
\begin{equation}\label{eq-NLS}
   u''+2u^2 v =0  \quad , \quad v''+2v^2 u =0 \ .
\end{equation}
In fact, it is easy to check that the commutator of these operators is the zero order operator
\begin{equation}
    [ L , B ]= \left[\begin{matrix}
    0& -u''-2u^2 v \\v''+2v^2 u & 0
    \end{matrix} \right],
\end{equation}
which is the zero operator by \eqref{eq-NLS}.

\medskip

Next we will study the spectral problem associated with the pair of operators $L, B$. This is the coupled eigenvalue problem 
\begin{equation}\label{eq-problem}
    ( L -\lambda  I_2)Y = \overline{0}
    \quad , \quad 
    ( B -\mu I_2)Y = \overline{0}  \ .
\end{equation}
The $\BC$-ideal in this case is described by the following result.

\begin{theorem}\label{thm-2}
Let us consider commuting MODOs $L$ and $B$ in $\mathcal{R}_{2}[D]$, with $L$ of order one and invertible leading coefficient. If $B\notin C[L]$ and $f(L,B)=\m0$, then $\BC (L,B)=(f)$, for $f(\lambda,\mu)=Dres(L-\lambda,B-\mu)$.
\end{theorem}
\begin{proof}
In the case of matrix coefficients of size $\ell=2$, the polynomial $f$ has degree $2$ in $\mu$, see Lemma \ref{lemma-1} . Hence the options are limited to: 
\begin{enumerate}
     \item By Theorem \ref{thm-ideal}, if $f$ is irreducible or has two different irreducible components then $\BC (L,B)=(f)$. 
     
    \item If $f$  is the square of a polynomial of type $h=\mu-R(\lambda)$, then by Corollary \ref{cor-trivial}, $\BC (L,B) =(h)$ if and only if $B=R(L)$. Hence, for $B$ not a polynomial in $L$, we obtain $\BC (L,B)=(f)$.
\end{enumerate}
\end{proof}

The classification of algebras $C[L,B]$ for MODOs of size $\ell=2$ follows from Theorem \ref{thm-2}: If $f$ has one irreducible component then, by Theorem \ref{thm-isocurve}, then $C[L,B]\simeq C[\lambda,\mu]/(f)$; If $f=h_1\cdot h_2$, by Theorem C, then $C[L,B]\simeq C[\lambda,\mu]/(h_1)\times C[\lambda,\mu]/(h_2)$.

\vspace{0.5cm}

The differential resultant $\Dres(L-\lambda,B-\mu)$ is the determinant of the $2\times 2$ matrix defined in \eqref{def-Dres} for $N_{\lambda}$,
\begin{align}
N_{\lambda}&=-A_1^{-1}(A_0-\lambda I_2)\\
M(L -\lambda, B-\mu)&=B_0-\mu I_2 +B_1 N_{\lambda}+B_2 (N_{\lambda}^2 +N'_{\lambda})=\\
&=
\left[\begin{matrix}
   -\imath u  v  +2
\,\imath{\lambda}^{2}-\mu&\imath u'  +2\,
u  \lambda\\ \noalign{\medskip}\imath v'  -2\,v  \lambda&\imath u  v  -2\,\imath{\lambda}^{2}-\mu 
\end{matrix}\right] .   
\end{align}
Moreover,
\begin{equation}\label{eq-spectralCurve}
    f(\lambda , \mu) = \Dres (L-\lambda,B-\mu)= {\mu}^{2}+
    4\,{\lambda}^{4}+ I_0 \lambda+ I_1 \ 
\end{equation}
where the differential polynomials $I_0 = u^{2}  v^{2}+ v'u'$ and $I_1 = -2\,i v' u  +2\,i u' v$ are first integrals of the NLS equation \eqref{eq-NLS}, since
\[
I'_0 = 2uu'v^2+2u^2vv' +v'' u' +vu'' =0  \ , \  I'_1 = -2\imath v'' u +2\imath u'' v=0.
\]
Consequently, the polynomial \eqref{eq-spectralCurve} defines a plane algebraic curve $\Gamma$ in $\mathbb{C}^2$, the spectral curve. Moreover, $f(L,B)$ is a MODO of zero order, equal to
\begin{equation}
    f(L,B)=\left[\begin{matrix}
   2v(2u^2v+u'') & 2(2u^2v+u'')'\\
   \noalign{\medskip}
   -2(2v^2u+v'')'&  2u(2v^2u+v'')
\end{matrix}\right], 
\end{equation}
which is the zero operator by \eqref{eq-NLS}.

\vspace{0.5cm}

Next we will study the spectral problem associated with the pair of operators $L, B$ at a point $P=(\lambda_0 ,\mu_0 )$. The spectral  problem  \eqref{eq-problem}
has associated the plane algebraic curve defined by $f(\lambda,\mu)=Dres(L-\lambda,B-\mu)$. Let us consider a point $P$ on the curve $\Gamma$  with $\mu_0 \not =0$ (i. e. a nonbranching point).  In particular, $P$ is a non singular point. Let $\cE$ be a Picard-Vessiot field for the system $DY=N_{\lambda_0 } Y$. Let $V_{\lambda_0 }$ be  kernel of $L -{\lambda_0 }$. It  is a $2$-dimensional $\mathcal{C}$-vector space. 
If we consider the operator $B -{\mu_0 }$ restricted to $V_{\lambda_0 }$, we obtain
\begin{equation}
\begin{array}{ll}
      (B -{\mu_0 })(\psi) =  & (B_0 -{\mu_0 } I)(\psi) +B_1 D(\psi) +B_2 D^2 (\psi)  =\\ &
      M(L-{\lambda_0 }, B-{\mu_0 })\cdot \psi \ .
\end{array}
\end{equation}
In addition, the matrix $M(L-{\lambda_0 }, B-{\mu_0 })$ has zero determinant since $f(P)=0$. Thus the linear map
\begin{equation}
   \xi: V_{\lambda_0 } \rightarrow V_{\lambda_0 }  \ , \ \xi (\psi ): = M(L-{\lambda_0 }, B-{\mu_0 })\cdot \psi ,
\end{equation}
has a nontrivial kernel $\mathcal{L}_P$,
\begin{equation}\label{eq-FiberBundel}
\mathcal{L}_P = \left\{ (\psi_1 ,\psi_2 ) \ : \ 
\left(-\imath u  v  +2
\,\imath{\lambda_0 }^{2}-{\mu_0 }\right) \psi_1 + \left( \imath u'  +2\,
u  {\lambda_0 } \right) \psi_2 =0 \ 
\right\} \ .
\end{equation}

This kernel $\mathcal{L}_P$ defines the one dimensional $C$-linear space of the common solutions of the linear differential systems $LY=\lambda_0 Y$ , $BY = \mu_0 Y$, and the rational function in the fraction field of $K[\lambda,\mu]/(f)$
\begin{equation}\label{eq-phi}
\phi (\lambda , \mu ,u,v) = -\frac{-\imath u  v  +2
\,\imath{\lambda}^{2}-\mu}{\imath u'  +2\,
u  \lambda}  \quad \textrm{satisfies } \phi_P : = \phi (P ,u,v)=\frac{\psi_2 }{\psi_1 }    
\end{equation}
as $(\lambda-\lambda_0 ,\mu-\mu_0 )/(f)$ is a maximal ideal of the ring $K[\lambda,\mu]/(f)$ . Moreover $\phi_P$ satisfies the Riccati-type equation $\phi_P'-u \phi_P^2 -2 \imath \lambda_0 \phi_P - v=0$, since
\[
\phi'-u \phi^2 -2 \imath \lambda \phi - v =-u \cdot f(\lambda ,\mu) \ ,
\]
in total agreement with \cite{GH}, formula  (3.62) with $p=v$ and $q=-u$. 

\begin{ex}\label{ex-irreducible}
If we consider $K=\mathbb{C}(e^{2\imath x})$ and the NLS potentials
$
u(x)=e^{-2\imath x}
$, $ v(x)= 2 e^{2\imath x} $, then
\[
L=\imath\left[\begin{matrix}D &e^{-2\imath x}\\2 e^{2\imath x}&-D\end{matrix}\right] \ , \ 
B= \imath\left[\begin{matrix}-2D^2-2&-2e^{-2\imath x}D+2\imath e^{-2\imath x}\\-4 e^{2\imath x}D-4\imath e^{2\imath x}&2D^2+2\end{matrix}\right] .
\]
Consequently, the spectral curve is 
$
 f(\lambda , \mu )= \mu^2 +4 (\lambda+1)^2 (\lambda^2-2\lambda+3) =0
$.  Its branching points are obtained for $\lambda_0 = -1 , 1+\imath \sqrt{2}, 1-\imath \sqrt{2}$. Observe that this curve is an irreducible  singular curve. The algebra $C[L,B]$ is isomorphic to the domain $\frac{C[\lambda , \mu ] }{(f)}$. The common solution of the coupled spectral problem \eqref{eq-problem}  at a nonbranching point $P=(\lambda_0 , \mu_0 )$ is
 \[
 \Psi= \begin{pmatrix}
 1 \\ \phi_P
 \end{pmatrix}\  \textrm{ with } \phi_P =\displaystyle -\frac{-2\imath +2\imath \lambda_0^2 -\mu_0 }{2+2\lambda_0 }\cdot  e^{2 \imath x }\ .
 \]
 \end{ex}

\begin{ex}\label{ex-nonirred}

Next consider $K=\mathbb{C}(x)$ and the NLS potentials
$u(x)=x$ and $v(x)=0$. Then $f(\lambda , \mu )= \mu^2 +4\lambda^4$ and $\phi=-\frac{2i\lambda^2-\mu}{i+2x\lambda}$. Observe that in this case the branching point is $P=(0,0)$.

We know that $f(L,B)=\m0$ but one can easily check that none of the irreducible components of $f$, namely $h_1 (\lambda,\mu )=\mu -2i\lambda^2$ nor $h_2 (\lambda,\mu )=\mu +2i\lambda^2$ are BC polynomials for the pair $L,B$. The decomposition \eqref{eq-decomposition} in Theorem C gives the ring structure of $C[L,B]\simeq C[\lambda ,\mu ]/(\mu^2 +4\lambda^4 ) $ as the product of $C[\lambda ,\mu ]/(h_i )$. In other words, for a polynomial $g\in \mathbb{C}[\lambda ,\mu]$ we have 
\[
g(L,B)=\m0 \Longleftrightarrow h_1 | g \quad \textrm{and}\quad 
h_2 | g .
\]
\end{ex}

\subsection*{Acknowledgements} 
S.L.~Rueda is partially supported by Research Group ``Modelos ma\-tem\'aticos no lineales''.
M.A.~Zurro is partially supported by Grupo UCM 910444. S.L.~Rueda and M.A.~Zurro partially supported by the grant PID2021-124473NB-I00, ``Algorithmic Differential Algebra and Integrability" (ADAI)  from the Spanish MICINN.

The authors would like to thank the anonymous referees for their insightful suggestions to improve the final version of this article.
\appendixpage

\appendix

\section{Differential Algebra Complements}\label{sec-DiffAlg}

Let $\mathbb{N}$ be the set of positive integers including $0$.
For concepts in differential algebra we refer the reader to \cite{VPS}, \cite{CHaj},  or \cite{Morales}.

A differential ring is a ring $R$ with a derivation $\partial$ on $R$. A differential ideal $I$ is an ideal of $R$ invariant under the derivation. A differential ring $R$ is called a simple differential ring if it has  no proper non zero differential ideals. We denote by
\begin{equation*}
    Const(R) = \{ r\in R \ | \ \partial (r) =0\} \ ,
\end{equation*}
which is called the ring of constants of $R$. 
Assuming that $R$ is a differential domain, its field of fractions $Fr (R)$ is a differential field with extended derivation 
$$\partial (f/g) = ( \partial (f)g-f\partial (g))/g^2.$$
A differential field $(K, \partial )$ is a differential ring which is a field. Given $a\in K$ we denote $\partial(a)$ by $a'$. Note that $Const(K)$ is a field whenever $K$ is. We assume that $C:=Const(K)$ has characteristic $0$.

Let $\alpha$ be an algebraic element over $K$. The derivation $\partial$ of $K$ can be extended to $K(\alpha )$, the minimum field generated by $\alpha$ and $K$, with extended derivation
\begin{equation}\label{der-alg-element}
\partial (\alpha ) = 
-\frac{P^{(d)}(\alpha )}{P' (\alpha )},
\end{equation}
where $P(T) \in K[T]$ is the minimal polynomial of $\alpha$, $P^{(d)}(T)$ denotes the polynomial obtained from $P(T)$ by derivating each one of its coefficients, and $P'(T)$ stands for the formal derivative of $P(T)$ with respect to the variable $T$. In particular, if $\alpha$ is algebraic over $C$ then $\alpha$ is a constant of $K$.

Moreover, we can extend the derivation $\partial$ of $K$  to an algebraic closure $\overline{K}^{alg}$ of $K$ using \eqref{der-alg-element}. Then, this algebraic closure has an algebraically closed field of constants, since
\begin{equation}\label{eq:ctes-alg-closed}
    Const( \overline{K}^{alg} ) = \overline{C}^{alg}.
\end{equation}
See for intastance \cite{Bron}, Corollary 3.3.1.

\medskip

For an ordinary differential system of equations, the above algebraic theory can be applied as follows. Let us fix $\ell\in\mathbb{N}$, $\ell\not=0$. {The derivation $\partial$ is extended to a derivation $D$ in the ring $M_{\ell}(K)$ of matrices with coefficients in $K$, as follows. Given $A=(a_{\alpha,\beta})\in M_{\ell}(K) $ then $D(A):=A'$, with $A'=(a'_{\alpha,\beta})$.}  Consider  an ordinary  differential system
\begin{equation}\label{eq-F-diff-sys}
    DY=AY  \ , \ \textrm{with } A \in {{M}_{\ell}\left( K\right)}\ ,
\end{equation}
and $ \ Y=(y_1 , \dots ,y_\ell )^t  $, $ \ DY=(y'_1 , \dots ,y'_\ell )^t  $.

{Let $R$ be a differential ring containing the differential field
$K$} and having ${C}$ as its field of constants.  A matrix $\Phi \in M_\ell (R )$
 is called \texttt{ a fundamental matrix for the equation} \eqref{eq-F-diff-sys} if  $\Phi$ is invertible and the equality $ D\Phi=A\Phi $
holds. Furthermore, if $\Phi $ and $\Psi$ are both fundamental matrices, then, applying the  derivation, we obtain that $\Delta = \Phi^{-1} \Psi$ is a constant matrix. Consequently, $\Phi = \Psi \Delta$ for a matrix $\Delta \in M_\ell ( {C} )$.

The following definition establishes the necessary requirements so that a differential field contains the solutions of the given differential system, and that it is the smallest differential field with this property keeping the field of constants fixed.

\begin{definition}
\texttt{A Picard-Vessiot ring over $K$ for the equation \eqref{eq-F-diff-sys}}, is a differential ring $R$ over $K$ satisfying:
\begin{enumerate}
    \item $R$ is a simple differential ring.
    \item There exists a fundamental matrix $\Psi$ for \eqref{eq-F-diff-sys} with coefficients in $R$, i.e., the matrix $\Psi \in  GL_\ell (R)$ satisfies $\Psi' = A\Psi$.
    \item $R$ is generated as a ring by $K$, the entries of a fundamental matrix $\Psi$ and the inverse of the determinant of $\Psi$.
\end{enumerate}
 Its fraction field $\Sigma$ is called the \texttt{ Picard-Vessiot field} of this differential system.
\end{definition}

\begin{remark}\label{rem-E-appendix}
Observe that any Picard-Vessiot ring $R$ for the equation \eqref{eq-F-diff-sys} is a domain, since $R$ has no proper maximal differential ideals. See \cite{VPS}, Proposition 1.20. Moreover, assuming that the field of constants $C$ is algebraically closed, a classical theory (Picard-Vessiot Theory) guaranties the existence and uniqueness of the Picard-Vessiot field for the equation \eqref{eq-F-diff-sys}. See \cite{VPS}, Proposition 1.22.

\end{remark}

\medskip

In this work we apply the previous considerations to the following framework.

\medskip

Let $\lambda$ and $\mu$ be  algebraic variables  with respect to $\partial$. Thus $\partial \lambda = 0$ and $\partial \mu = 0$. The derivation $\partial$ of $K$ can be extended to the polynomial ring $K[\lambda, \mu]$, and then $(K[\lambda, \mu], \partial)$ is a differential ring whose ring of constants is $(C[\lambda, \mu], \partial)$. 

We define $\mathcal{F}$ to be the differential field $\mathcal{F}=Fr(K[\lambda, \mu ]) = K(\lambda, \mu )$ and $\overline{\mathcal{F}}$ an algebraic closure of $\mathcal{F}$. By formula \eqref{eq:ctes-alg-closed} applied to the differential field  $\mathcal{F}$, the field of constants     of $\overline{\mathcal{F}} $ is algebraically closed and equal to
\begin{equation}\label{eq-constants-cl-appenddix}
      \mathcal{C} :=\overline{Const (\mathcal{F} )}^{alg}  .
\end{equation}

\medskip

Given a differential operator $L$  in $M_{\ell}(K)[D] $, 
we consider the spectral problem
\begin{equation*}
    LY=\lambda  Y\quad , 
    \quad  Y=(y_1,\ldots ,y_{\ell})^t .
\end{equation*}
This spectral problem can be studied   for $L=A_0+A_1 D$,
assuming that $A_1$ is invertible.
Let $N_{\lambda}$ be the matrix  $N_{\lambda}=-A_1^{-1}(A_0-I_{\ell}\lambda )$ in $M_{\ell} (K[\lambda , \mu ])$, and consider  the differential system
\begin{equation}\label{eq-PV-system_lambda-appendix}
    DY= N_\lambda Y \quad \textrm{ with } \ N_{\lambda} \in M_\ell (\overline{\mathcal{F}} ).
\end{equation}
Let $\mathcal{E}$ be a Picard-Vessiot extension of $\overline{\mathcal{F}}$ with field of constants $\mathcal{C}$, for the differential system \eqref{eq-PV-system_lambda-appendix}, see Remark \ref{rem-E-appendix}. Then, there exits a fundamental matrix $\Psi$  such that 
\begin{equation}
D\Psi = N_\lambda \Psi    
\quad , \quad 
\textrm{ with } \Psi \in M_{\ell} (\mathcal{E}).
\end{equation}
Furthermore, the polynomial ring $C[\lambda , \mu ]$ can be recovered from $\mathcal{C} $ as read in formula \eqref{eq-cts-basicas}.

\begin{lemma}\label{lemma-para-Previato}
The polynomial ring $C[\lambda , \mu ]$ equals
\begin{equation}\label{eq-cts-basicas}
 C[\lambda , \mu ] =    K[\lambda, \mu ] \cap \mathcal{C} \ .
\end{equation}

\end{lemma}

\begin{proof}
Obviously we have  $C[\lambda , \mu ] \subset    K[\lambda, \mu ] \cap \mathcal{C}$. But the converse is also true since the ring of constants of $K[\lambda, \mu ] $ is $C[\lambda , \mu ] $. Consequently we obtain the required equality.
\end{proof}

\bibliographystyle{acm}

\bibliography{bibliography.bib}

\begin{thebibliography}{10}

\bibitem{AKNS1974}
{\sc Ablowitz, M., Kaup, D., Newell, A., and Segur, H.}
\newblock The inverse scattering transform - {F}ourier analysis for nonlinear
  problems.
\newblock {\em Stud. Appl. Math. 53\/} (1974), 249--315.

\bibitem{AM}
{\sc Atiyah, M.~F., and Macdonald, I.~G.}
\newblock {\em Introduction to Commutative Algebra}.
\newblock Addison-Wesley, 1969.

\bibitem{Baker1928}
{\sc Baker, H.~F.}
\newblock Note on ``{Commutative} ordinary differential operators, by {J}. {L}.
  {Burchnall} and {T}. {W}. {Chaundy}.''.
\newblock {\em Proc. R. Soc. Lond., Ser. A 118\/} (1928), 584--593.

\bibitem{Bron}
{\sc Bronstein, M.}
\newblock {\em Symbolic integration I: transcendental functions}, vol.~1.
\newblock Springer Science \& Business Media, 2013.

\bibitem{BC1}
{\sc Burchnall, J., and Chaundy, T.}
\newblock Commutative ordinary differential operators.
\newblock {\em Proc. R. Soc. A 118\/} (1928), 557--583.

\bibitem{BC2}
{\sc Burchnall, J., and Chaundy, T.}
\newblock Commutative ordinary differential operators {II}. the {I}dentity
  ${P}^n = {Q}^m$.
\newblock {\em Proc. R. Soc. A 134\/} (1931), 471--485.

\bibitem{Ch}
{\sc Chardin, M.}
\newblock Differential resultants and subresultants.
\newblock {\em Proc. FCT'91. Lecture Notes in Comput. Sci., Springer-Verlag
  529\/} (1991), 471--485.

\bibitem{ChurchillKovasic}
{\sc Churchill, R.~C., and Kovacic, J.~J.}
\newblock Cyclic vectors.
\newblock In {\em Proceedings of the International Workshop. Differential
  Algebra and Related Topics\/} (New York, NY, 2002), World Scientific,
  pp.~191--218.

\bibitem{CHaj}
{\sc Crespo, T., and Hajto, Z.}
\newblock {\em Algebraic groups and differential {Galois} theory}, vol.~122 of
  {\em Grad. Stud. Math.}
\newblock American Mathematical Society, 2011.

\bibitem{Dubrovin}
{\sc Dubrovin, B. A.~E.}
\newblock Completely integrable hamiltonian systems associated with matrix
  operators and abelian varieties.
\newblock {\em Funct. Anal. Appl. 11}, 4 (265-277), 51--108.

\bibitem{GH}
{\sc Gesztesy, F., and Holden, H.}
\newblock {\em Soliton equations and their algebro-geometric solutions:
  {Volume} 1, $(1+1)$-dimensional continuous models}, vol.~79 of {\em Cambridge
  Stud. Adv. Math.}
\newblock Cambridge University Press, 2003.

\bibitem{Good}
{\sc Goodearl, K.}
\newblock Centralizers in differential, pseudo-differential and fractional
  differential operator rings.
\newblock {\em Rocky Mountain J. Math. 13}, 4 (1983), 573--618.

\bibitem{Grinevich1987}
{\sc Grinevich, P.~G.}
\newblock Vector rank of commuting matrix differential operators. {P}roof of
  {S}. {P}. {N}ovikov{\textquotesingle}s criterion.
\newblock {\em Mathematics of the {USSR}-Izvestiya 28}, 3 (jun 1987), 445--465.

\bibitem{GZ}
{\sc Guo, J., and Zheglov, A.~B.}
\newblock On some questions around {Berest's} conjecture.
\newblock arXiv:2203.13343.

\bibitem{Kasman2017}
{\sc Kasman, A.}
\newblock On factoring an operator using elements of its kernel.
\newblock {\em Comm. Algebra 45}, 4 (2017), 1443--1451.

\bibitem{KasmanPreviato2001}
{\sc Kasman, A., and Previato, E.}
\newblock Commutative partial differential operators.
\newblock {\em Physica D: Nonlinear Phenomena 152-153\/} (2001), 66--77.
\newblock Advances in Nonlinear Mathematics and Science: A Special Issue to
  Honor Vladimir Zakharov.

\bibitem{KasmanPreviato2010}
{\sc Kasman, A., and Previato, E.}
\newblock Factorization and resultants of partial differential operators.
\newblock {\em Math. Comput. Sci. 4\/} (2010), 169–184.

\bibitem{Katz1987}
{\sc Katz, N.~M.}
\newblock A simple algorithm for cyclic vectors.
\newblock {\em American Journal of Mathematics 109}, 1 (1987), 65--70.

\bibitem{MuKi}
{\sc Kimura, M., and Mulase, M.}
\newblock Commutative algebras of ordinary differential operators with matrix
  coefficients.
\newblock 1996.

\bibitem{Kolchin}
{\sc Kolchin, E.~R.}
\newblock {\em Differential algebra and algebraic groups}.
\newblock No.~54 in Pure Appl. Math. (Amst.). Academic Press, Boston, MA, 1973.

\bibitem{K78}
{\sc Krichever, I.}
\newblock Commutative rings of ordinary linear differential operators.
\newblock {\em Func. Anl. Applic. 12}, 3 (1978), 175--185.

\bibitem{krichever}
{\sc Krichever, I.}
\newblock Rational solutions of {K}adomtsev-{P}etviashvili equation and
  integrable systems of n particles on a line.
\newblock {\em Func. Anl. Applic. 12\/} (1978), 59--61.

\bibitem{LY}
{\sc Li, W., and Yuan, C.}
\newblock Elimination theory in differential and difference algebra.
\newblock {\em J. Syst. Sci. Complex\/} (2019), 287--316.

\bibitem{MuLi}
{\sc Li, Y., and Mulase, M.}
\newblock Category of morphisms of algebraic curves and a characterization of
  {P}rym varieties.

\bibitem{McW}
{\sc McCallum, S., and Winkler, F.}
\newblock Resultants: Algebraic and differential.
\newblock {\em Techn. Rep. J. Kepler University RISC18-08\/} (2018).

\bibitem{MR1999}
{\sc Morales-Ruiz, J.~J.}
\newblock {\em Differential {G}alois theory and non-integrability of
  {H}amiltonian systems}.
\newblock Progr. Math. Birkh\"auser, Basel, 1999.

\bibitem{Morales}
{\sc Morales-Ruiz, J.~J.}
\newblock {\em Differential {Galois} theory and non-integrability of
  {Hamiltonian} systems}, vol.~179 of {\em Progress in Mathematics}.
\newblock Birkh\"auser, 1999.

\bibitem{MRZ1}
{\sc Morales-Ruiz, J.~J., Rueda, S., and Zurro, M.}
\newblock {Factorization of KdV Schr\" odinger operators using differential
  subresultants}.
\newblock {\em Adv. Appl. Math. 120\/} (2020), 102065.

\bibitem{MRZ2}
{\sc Morales-Ruiz, J.~J., Rueda, S.~L., and Zurro, M.~A.}
\newblock Spectral {Picard{\textendash}Vessiot} fields for {Algebro-geometric}
  {Schr\"odinger} operators.
\newblock {\em Ann. Inst. Fourier (Grenoble) 71}, 3 (2021), 1287--1324.

\bibitem{Oganesyan2017}
{\sc Oganesyan, V.}
\newblock {Matrix Commuting Differential Operators of Rank 2 and Arbitrary
  Genus}.
\newblock {\em Int. Math. Res. Notices 2019}, 3 (07 2017), 834--851.

\bibitem{previato}
{\sc Previato, E.}
\newblock Hyperelliptic quasi-periodic and soliton solutions of the nonlinear
  {S}chrödinger equation.
\newblock {\em Duke Math. J. 52}, 2 (1985), 329--377.

\bibitem{Prev}
{\sc Previato, E.}
\newblock Another algebraic proof of {W}eil's reciprocity.
\newblock {\em Atti Accad. Naz. Lincei Cl. Sci. Fis. Mat. Natur. Rend. Lincei
  (9) Mat. Appl 2}, 2 (1991), 167--171.

\bibitem{Previato2008}
{\sc Previato, E.}
\newblock Multivariable {B}urchnall–{C}haundy theory.
\newblock {\em Philos. Trans. Roy. Soc. A. 366\/} (2008), 1155--1177.

\bibitem{PRZ2019}
{\sc Previato, E., Rueda, S.~L., and Zurro, M.~A.}
\newblock Commuting {O}rdinary {D}ifferential {O}perators and the {D}ixmier
  {T}est.
\newblock {\em {SIGMA} {Symmetry} {Integrability} {Geom.} {Methods} {Appl.}
  15}, 101 (2019), 23 pp.

\bibitem{Ritt}
{\sc Ritt, J.~F.}
\newblock {\em Differential algebra}, vol.~33.
\newblock Amer. Math. Soc. Colloq. Publ., 1950.

\bibitem{VPS}
{\sc van~der Put, M., and Singer, M.~F.}
\newblock {\em Galois theory of linear differential equations}, vol.~328 of
  {\em Grundlehren Math. Wiss.}
\newblock Springer, 2012.

\bibitem{Ve}
{\sc Verdier, J.~L.}
\newblock \'{E}quations diff\' erentielles algebriques.
\newblock {\em Lecture Notes in Math. 710\/} (1979), 101--122.

\bibitem{wilson}
{\sc Wilson, G.}
\newblock Commuting flows and conservation laws for {L}ax equations.
\newblock {\em Math. Proc. Camb. Phil. Soc. 86\/} (1979), 131--143.

\bibitem{ZS}
{\sc Zakharov, V., and Shabat, A.~B.}
\newblock A scheme for integrating the non-linear equations of mathematical
  physics by the inverse scattering method.
\newblock {\em Funct. Anal, and Its Appl. 8\/} (1974), 43--53 (Russian),
  226--235 (English).

\end{thebibliography}
\end{document}